\documentclass[letterpaper,10pt,conference]{ieeeconf}
\usepackage{amsmath}
\usepackage{amssymb}
\usepackage{subcaption}
\usepackage{hyperref}
\let\labelindent\relax

\usepackage[backgroundcolor=blue!10,bordercolor=blue!70!black,textsize=scriptsize]{todonotes}
\usepackage[shortlabels]{enumitem}
\usepackage{graphicx}
\usepackage{url}
\usepackage{color}
\usepackage{cite}
\usepackage{epstopdf}
\usepackage{tikz}
\usetikzlibrary{shapes.misc,decorations.pathreplacing,decorations.markings,patterns,arrows.meta}

\newcommand{\NN}{{\mathbb N}}
\newcommand{\RR}{{\mathbb R}}

\newcommand{\ZZ}{{\mathbb Z}}
\newcommand{\cA}{{\mathcal{A}}}

\newcommand{\cM}{{\mathcal{M}}}

\newcommand{\cU}{{\mathcal{U}}}

\newcommand{\cX}{{\mathcal{X}}}

\newcommand{\argmin}{\mathop{\mathrm{argmin}}}

\newcommand{\mindeg}{\mathop{\mathrm{mindeg}}}

\newtheorem{theorem}{Theorem}
\newtheorem{proposition}{Proposition}
\newtheorem{remark}{Remark}
\newtheorem{game}{Game}

\newtheorem{example}{Example}

\newtheorem{algorithm}{Algorithm}

\IEEEoverridecommandlockouts  

\newcommand{\rulesep}{\unskip\ \vrule width 1pt\ }

\title{\Large \bf Graph-Based Controller Synthesis for Safety-Constrained, Resilient Systems}
\author{Matija~Buci\'{c}, Melkior~Ornik, and Ufuk~Topcu
\thanks{M.~Buci\'{c} is with ETH Z\"{u}rich. M.~Ornik and U.~Topcu are with the University of Texas at Austin. E-mails: \{matija.bucic@math.ethz.ch, mornik@ices.utexas.edu, utopcu@utexas.edu\}}
\thanks{This work was supported in part by grants FA8750-17-C-0087 from the Defense Advanced Research Projects Agency and 200021-175573 from the Swiss National Science Foundation.}
}

\begin{document}
\maketitle

\begin{abstract}
Resilience to damage, component degradation, and adversarial action is a critical consideration in design of autonomous systems. In addition to designing strategies that seek to prevent such negative events, it is vital that an autonomous system remains able to achieve its control objective even if the system partially loses control authority. While loss of authority limits the system's control capabilities, it may be possible to use the remaining authority in such a way that the system's control objectives remain achievable. In this paper, we consider the problem of optimal design for an autonomous system with discrete-time linear dynamics where the available control actions depend on adversarial input produced as a result of loss of authority. The central question is how to partition the set of control inputs that the system can apply in such a way that the system state remains within a safe set regardless of an adversarial input limiting the available control inputs to a single partition elements. We interpret such a problem first as a variant of a safety game, and then as a problem of existence of an appropriate edge labeling on a graph. We obtain conditions for existence and a computationally efficient algorithm for determining a system design and a control policy that preserve system safety. We illustrate our results on two examples: a damaged autonomous vehicle and a method of communication over a channel that ensures a minimal running digital sum.
\end{abstract}

\section{Introduction}
\label{secmot}
Controller's loss of authority over parts of an autonomous system may happen in many scenarios:
\begin{enumerate}[(a)]
\item \textit{System damage and component degradation.} An autonomous system operating for substantial periods of time in remote, unknown, or hostile environment will inevitably sustain damage or experience partial system failures over time due to malfunctions. Examples include unmanned aerial vehicles (UAVs) operating over contested territory \cite{RatSen04}, search-and-rescue robots \cite{Chaetal18}, and rovers performing missions on extraterrestrial surfaces \cite{Wasetal99}.
\item \textit{Hostile takeover.}
In a number of adversarial settings, the adversary will attempt to take over elements of the system and disturb its regular functions. A typical setting is that of attacks on computer networks \cite{Vat01} and power systems \cite{AmiGia12,Zhuetal14}, where, because of the vastness of the network and heterogeneity and physical distance between system elements, an adversarial agent may be able to penetrate a part of the system. Hostile takeover scenarios also include recent successful attacks resulting in loss of control over UAVs; see, e.g., \cite{Woletal14,HarGil16}.

\item \textit{User-responsive systems.} Settings where an automated controller is required to respond to (a priori unknown) user inputs in a particular way necessarily yield a part of the control authority to the user. Such scenarios include resource distribution in parallel computing \cite{Feietal90}, semantic web service composition \cite{Rodetal12}, and communication protocols \cite{Sha48}.
\end{enumerate}

In all of the above settings, it is critical to ensure that the autonomous system can perform its tasks regardless of external inputs that may affect the system. A standard method of ensuring continued functioning of the system is through imposing redundancy or near-redundancy in design. For instance, critical components in commercial airplanes are duplicated \cite{Dow09}, and military UAVs use a combination of different sensing systems for navigation \cite{HarGil16}. In the latter example, while these different sensing systems do not work in the same way and, in regular flight regime, serve to complement each other, each system is able to ensure that the UAV can achieve basic control objectives even if complementary systems are not functioning. 

Motivated by the above scenarios, our work seeks to investigate how to guarantee continued safe operation of an abstract control system in which some components are no longer under the controller's authority. We focus on systems with linear driftless discrete-time dynamics, and interpret the partial loss of control authority as limitations on the controller's choice of actions, based on adversarial inputs. The control objective that we investigate is {\em safety}: the system state is required to remain within a particular set throughout the system run. We are interested in (i) developing a safe control policy, if one exists, and (ii) determining a {\em resilient system design} --- i.e., a partition of the set of all control inputs that the system can apply --- which ensures that the system will be able to remain safe even if the adversary limits the available actions to a single element of the partition at any given time.

The work in this paper is closely related to previous research on control of safety-critical systems \cite{TomLyg00,Tab09} and safety games \cite{Beretal02, DoyRas11, NelTop16}. In particular, as we will show, given a system design, i.e., possible control inputs given an adversarial input, a safe control policy can be interpreted as a winning strategy for a turn-based safety game. This interpretation leads to a computationally efficient algorithm for designing a safe control policy. However, such an algorithm does not directly provide for a computationally feasible procedure of determining whether there {\em exists} a resilient system design, as each design corresponds to a different safety game, and searching through all possible games is computationally prohibitive. We address this challenge through a method based on a graph-theoretical interpretation of system design.

The outline of the remainder of this paper is as follows. In Section \ref{probsta} we provide a motivation for theoretical framework used in the paper, and formally describe the problems of safe control design and resilient system design under adversarial action. We then interpret such problems within the context of safety games in Section \ref{gam}, resulting in a simple solution to the problem of safe control design. We interpret the problem of resilient system design in a graph-theoretical setting in Section \ref{gra}, and --- using the probabilistic method, as described in \cite{AloSpe08} --- provide a sufficient condition and a necessary condition for its solvability in Section \ref{condit}. Based on the previous section, we provide a computationally efficient algorithm for resilient system design and construction of a safe control policy in Section \ref{ela}. Section \ref{exams} illustrates our techniques on two examples: an autonomous vehicle experiencing partial loss of control authority, and design of codes for communication over a channel with a bounded running digital sum.

\textbf{Notation.} The symbol $\NN$ denotes all strictly positive integers, $\NN_0$ denotes all nonnegative integers, and $\ZZ$ denotes all integers. For $m\in\NN$, $[m]$ denotes the set $\{1,\ldots,m\}$. For a set $\cX$, $|\cX|$ denotes its cardinality, and $2^{\cX}$ the set of all its subsets. For an event $B$ within a particular probability distribution, $\Pr(B)$ denotes the probability of $B$ occurring. For a graph $G=(V,E)$ and vertex $v\in V$, $\deg_G(v)$ denotes the (outgoing, if the graph is directed) degree of $v$, and $\mindeg(G)$ denotes the minimal (outgoing) degree of any vertex in $V$. If $G, H$ are graphs, $G\subseteq H$ signifies that $G$ is an induced subgraph of $H$. Vector $e_i$ denotes the standard basis vector consisting solely of zeros, except for a $1$ in the $i$-th position. Symbol $\|v\|_\infty$ denotes the max-norm of a vector $v\in\RR^n$, and $\|v\|_1$ denotes the $1$-norm of a vector $v$.


\section{Problem Statement}
\label{probsta}

Consider a system operating with discrete-time dynamics \begin{equation}
\label{dtd}
x(t+1)=x(t)+u(t)
\end{equation}
for all times $t\in\NN_0$, where $x(0)\in\ZZ^n$ and $u\in\cU\subseteq\ZZ^n$, with a finite $\cU$.
While model \eqref{dtd} is simple, our use of it is motivated by its wide presence in robotic exploration (see, e.g., \cite{Yam97, Megetal12, Oswetal16}, and the references therein) as well as its use in communication over a channel \cite{CohLit91}. As we will discuss in subsequent sections, \eqref{dtd} yields a straightforward graph-theoretical interpretation of system motion which may lead to generalizations for more complex models. 

To provide motivation for the problems that we will pose, let us assume that dynamics \eqref{dtd} represent an autonomous system controlled by actuators $A_1$, $A_2$, \ldots, $A_p$. The control effort $u(t)$ is then given as $u(a_1(t),\ldots,a_p(t))$, where $a_i(t)\in\cA_i$ is the setting of actuator $A_i$ at time $t$, and $\cU=\{u(a_1,\ldots,a_p)~|~a_i\in\cA_i, i=1,\ldots,p\}$.

We are interested in the scenario where the controller experiences loss of authority over some of the actuators, say $A_1, \ldots, A_r$. Thus, the choice of $a_1(t), \ldots, a_r(t)$ is not made by the controller, and any control actuation $u(t)$ needs to chosen in the set 
$U(a_1(t),\ldots,a_r(t))= \{u(a_1(t),\ldots,a_r(t),a_{r+1},\ldots, a_p)~|~a_i\in \cA_i, i\geq r+1\}$.
We assume that we do not possess any prior knowledge about the inputs $a_1(t), \ldots, a_r(t)$; these may be subjects to adversarial choices. 

The control objective that we consider is {\em safety}. That is, we want to ensure that $x(t)\in S$ for all $t\geq 0$, where $S\subseteq\ZZ^n$ is a predetermined set with $x(0)\in S$. We are interested in two questions:
\begin{enumerate}[(i)]
\item For given sets $U(a_1,\ldots, a_r)$, determine, if it exists, a control policy that guarantees system safety regardless of choices $a_1(t), \ldots, a_r(t)$.
\item Design sets $U(a_1, \ldots, a_r)$ so that the above control policy exists. 
\end{enumerate}

The latter question corresponds to designing the abilities and role of each actuator in such a way that the system is resilient to loss of authority over some of the actuators.

If the system can exhibit perfect redundancy, i.e., $U(a_1, \ldots, a_r)=\cU$ for every $a_1\in\cA_1$, \ldots, $a_r\in\cA_r$, questions (i) and (ii) are simple. However, redundancy is often undesirable due to cost, weight, or resource consumption \cite{Sghetal08}. Thus, we assume that it is impossible to execute exactly the same control with two different actuations. Under this assumption, $\{U(a_1,\ldots,a_r)~|~ a_1\in\cA_1, \ldots, a_r\in\cA_r\}$ is a partition of $\cU$. 
For the sake of simpler notation, we denote $\cA_1\times\cA_2\times\cdots\times\cA_r=[m]$ for some $m\in\NN$.

Questions (i) and (ii) are now formulated as follows.

{\em Restricted partition control problem (RPCP):}
Let $S\subseteq\ZZ^n$ and $x(0)\in S$. Let $\cU\subseteq\ZZ^n$ be finite, and $U:[m]\to 2^\cU$ such that $\{U(1),\ldots,U(m)\}$ is a partition of $\cU$. Does there exist a function $\hat{u}:\cup_{i=1}^\infty [m]^i\to\cU$ such that
\begin{enumerate}[(i)]
\item $\hat{u}(d_1,\ldots,d_k)\in U(d_k)$ for all $d_1,\ldots,d_k\in [m]$, and
\item for every $d:\NN_0\to[m]$, if $x(t)$ is the solution of \eqref{dtd} with $u(t)=\hat{u}(d(0),\ldots,d(t))$, then $x(t)\in S$ for all $t\in\NN_0$?
\end{enumerate}

{\em Free partition control problem (FPCP):}
Let $S\subseteq\ZZ^n$ and $x(0)\in S$. Let $\cU\subseteq\ZZ^n$ be finite. Does there exist a partition $\{U(1),\ldots,U(m)\}$ for which the RPCP admits a solution?

We note that in practice the available choices of partitions in the FPCP may be subject to constraints, e.g., physical limitations in design of actuators. We use the unconstrained version to provide an elegant illustration of a general approach to solving the above problems. Before moving towards solutions of the RPCP and the FPCP, let us introduce a running example.

\begin{example}[Damaged vehicle]
\label{exrun}
Consider an autonomous vehicle moving on $\ZZ^2$ according to dynamics \eqref{dtd}. At every instance in time, the vehicle can perform one of five actions: go one position to the north, south, east or west, or remain in the same position. In other words, $\cU=\{e_1,-e_1,e_2,-e_2,(0,0)\}$. The vehicle's initial position is given by $x(0)=(0,0)$, and the safe set $S$ is given by $S=\{x\in\ZZ^2~|~\|x\|_\infty\leq 1\}$. The setup is graphically illustrated in Fig.~\ref{runn}.

Let us first consider the RPCP with $m=2$ and $U(1)=\{e_1,e_2\}$, $U(2)=\{-e_1,-e_2,(0,0)\}$. In such a case, the RPCP does not admit a solution. For instance, if the adversary continually chooses $d=1$, the vehicle will have to keep moving north or east. Hence, after no more than $3$ steps, it will be forced to leave $S$. This situation is shown on the left side of Fig.~\ref{runn}.

\begin{figure}[ht]
\center
\begin{tikzpicture}[scale=0.95]
\clip (-1.65,-1.65) rectangle (2.65,2.65);
\tikzstyle{point}=[thick,fill=black, draw=black,circle,inner sep=0pt,minimum width=7pt,minimum height=7pt]
\tikzstyle{darkstyle}=[circle,draw,fill=gray!40,inner sep=0pt,minimum width=15pt,minimum height=15pt]
\tikzstyle{greenstyle}=[circle,draw,fill=green!20,inner sep=0pt,minimum width=15pt,minimum height=15pt]
\tikzset{->-/.style={decoration={
  markings,
  mark=at position .5 with {\arrow{>}}},postaction={decorate}}}
\tikzset{-<-/.style={decoration={
  markings,
  mark=at position .5 with {\arrow{Computer Modern Rightarrow[slant=.6]}}},postaction={decorate}}}

\foreach \x in {-3,...,3}
    \foreach \y in {-3,...,3}
       \node (\x\y) at (\x,\y) {};
       
\foreach \x in {-3,...,3}
    \foreach \y in {-3,...,3} {
    	\draw[red,->-] (\x\y) arc(-45:45:0.71);
    	\draw[blue,->-] (\x\y) arc(45:135:0.71);     
     	\draw[blue,->-] (\x\y) arc(135:225:0.71);
         \draw[red,->-] (\x\y) arc(225:315:0.71);
    }
\foreach \x in {-1,...,2}
    \foreach \y in {-1,...,2} {                 \draw[blue,-<-] (\x,\y+0.2) arc(-90:270:0.13);
    }

\foreach \x in {0,...,1}
    \foreach \y in {0,...,1} {   
\draw[ultra thick, red, ->-] (\x\y) arc(225:315:0.71);
\draw[ultra thick, red, ->-] (\x\y) arc(-45:45:0.71);
}

\draw[ultra thick, red, ->-] (0,2) arc(225:315:0.71);
\draw[ultra thick, red, ->-] (0,2) arc(-45:45:0.71);
\draw[ultra thick, red, ->-] (2,0) arc(225:315:0.71);
\draw[ultra thick, red, ->-] (2,0) arc(-45:45:0.71);
 
\foreach \x in {-2,...,2}
    \foreach \y in {-2,...,2}
       \node[darkstyle] (\x\y) at (\x,\y) {};

\foreach \x in {-1,...,1}
    \foreach \y in {-1,...,1}
       \node[greenstyle] (\x\y) at (\x,\y) {};
 
\node[point] at (0,0) {};
\end{tikzpicture}
\hspace*{-2.5pt}
\rulesep
\begin{tikzpicture}[scale=0.95]
\clip (-1.65,-1.65) rectangle (2.65,2.65);
\tikzstyle{point}=[thick,fill=black, draw=black,circle,inner sep=0pt,minimum width=7pt,minimum height=7pt]
\tikzstyle{darkstyle}=[circle,draw,fill=gray!40,inner sep=0pt,minimum width=15pt,minimum height=15pt]
\tikzstyle{greenstyle}=[circle,draw,fill=green!20,inner sep=0pt,minimum width=15pt,minimum height=15pt]
\tikzstyle{darkgreen}=[circle,draw,fill=green!70!black,inner sep=0pt,minimum width=15pt,minimum height=15pt]
\tikzset{->-/.style={decoration={
  markings,
  mark=at position .5 with {\arrow{>}}},postaction={decorate}}}
\tikzset{-<-/.style={decoration={
  markings,
  mark=at position .5 with {\arrow{Computer Modern Rightarrow[slant=.6]}}},postaction={decorate}}}

\foreach \x in {-3,...,3}
    \foreach \y in {-3,...,3}
       \node (\x\y) at (\x,\y) {};
 
\foreach \x in {-3,...,3}
    \foreach \y in {-3,...,3} {
    	\draw[blue,->-] (\x\y) arc(-45:45:0.71);
    	\draw[red,->-] (\x\y) arc(45:135:0.71);     
     	\draw[blue,->-] (\x\y) arc(135:225:0.71);
         \draw[red,->-] (\x\y) arc(225:315:0.71);
   }
\foreach \x in {-1,...,2}
    \foreach \y in {-1,...,2} {
         \draw[blue,-<-] (\x,\y+0.2) arc(-90:270:0.13);
    }
    
\draw[thick, blue,-<-] (0,0.2) arc(-90:270:0.13);
\draw[thick, blue,-<-] (1,0.2) arc(-90:270:0.13);
    
\draw[ultra thick, blue] (0,0.2) arc(-90:270:0.13);
\draw[ultra thick, blue] (1,0.2) arc(-90:270:0.13);
\draw[ultra thick, red,->-] (0,0) arc(225:315:0.71);
\draw[ultra thick, red,->-] (1,0) arc(45:135:0.71);

\foreach \x in {-2,...,2}
    \foreach \y in {-2,...,2}
       \node[darkstyle] (\x\y) at (\x,\y) {};
 
\foreach \x in {-1,...,1}
    \foreach \y in {-1,...,1}
       \node[greenstyle] (\x\y) at (\x,\y) {};
 
\node[darkgreen] at (0,0) {};
\node[point] at (0,0) {};
\node[darkgreen] at (1,0) {};
\end{tikzpicture}
\caption{The picture on the left illustrates a counterexample to solvability of the RPCP for $U(1)=\{e_1,e_2\}$, $U(2)=\{-e_1,-e_2,(0,0)\}$ in Example \ref{exrun}. The safe set $S$ is denoted in light green. The vehicle's initial position $x(0)=(0,0)$ is denoted by a black circle. Possible vehicle movements from each $x\in\ZZ^n$ are denoted by an arrow. Red arrows denote available movements when the adversary chooses $d=1$, and blue arrows denote available movements when $d=2$. Possible movements in the case when the adversary chooses $d(0)=d(1)=d(2)=1$ are drawn thickly. The picture on the right illustrates of solvability of the FPCP in Example \ref{exrun}. Same notation as in the left is used. The partition $\{U(1),U(2)\}$ is chosen in such a way that, regardless of the choice of $d(t)$, the vehicle can always remain in the dark green subset of the safe set.}
\label{runn}
\end{figure}
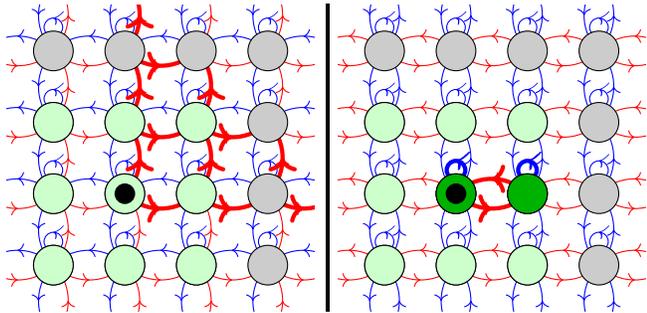

On the other hand, the FPCP admits a solution for $m=2$. Let $U(1)=\{e_1,-e_1\}$ and $U(2)=\{e_2,-e_2,(0,0)\}$. Then, when the adversary chooses $d=1$ for the first time, the vehicle can choose to move east, then west the next time, then east again, etc. If the adversary chooses $d=2$, the vehicle can remain in place. Hence, the vehicle will always remain within $S$. Such a strategy is depicted on the right side of Fig.~\ref{runn}.
\hfill $\square$
\end{example}

We now continue towards providing a solution for the RPCP and the FPCP.

\section{Game Formulation}
\label{gam}

The RPCP can be easily formulated as the question of existence of a winning strategy in the following two-player game.

\begin{game}
\label{gm1}
Let $S\subseteq\ZZ^n$ and $x(0)\in S$. Let $\cU\subseteq\ZZ^n$ be finite, and $\{U(1),\ldots,U(m)\}$ be a partition of $\cU$. Let $G=(V,E)$ be a graph with $V=\ZZ^n$ and $E=\{(x,y)~|~y-x\in\cU\}$, and $l:E\to[m]$ a labeling given by 
\begin{equation}
\label{labe}
l(x,y)=d \qquad\textrm{if }y-x\in U(d)\textrm{.}
\end{equation}

The game proceeds as follows. Before time $t=0$, a token is placed at $x(0)$. At every time step $t$, Player 1 first chooses an element $d\in [m]$. Then, Player 2 chooses an element $x(t+1)\in V$ such that $(x(t),x(t+1))\in E$ and $l(x(t),x(t+1))=d$, if such an element exists, and moves the token to $x(t+1)$. The game now proceeds to the next time step. Player 2 wins the game if it can always move the token, and the token remains within $S$ for all $t\in\NN_0$. Otherwise, Player 1 wins.
\end{game}

\begin{proposition}
\label{propo0}
The RPCP admits a solution if and only if there exists a winning strategy for Player 2 in Game \ref{gm1}.
\end{proposition}
\begin{proof}
By taking $u(t)=x(t+1)-x(t)$, it is clear that the movement of the token in Game \ref{gm1} corresponds to \eqref{dtd}. The requirement that $(x(t),x(t+1))\in E$ and $l(x(t),x(t+1))=d$ corresponds to the requirement that $u(t)\in U(d(t))$. Thus, Player 2 has a winning strategy in Game \ref{gm1} if and only if there exists $u(t)\in U(d(t))$, possibly dependent on all previous inputs $d(0),\ldots,d(t)$, such that $x(t)\in S$. The latter statement is exactly the statement of the RPCP.
\end{proof}

Game \ref{gm1} is a turn-based safety/reachability game with complete information as described \cite{DoyRas11}. Thus, for finite $S$, the RPCP can be solved in linear time with respect to the size of $S$ \cite{DoyRas11}. In the remainder of this paper, we focus on the FPCP. In a game-theoretical setting, the FPCP can be posed as follows.

\begin{game}
\label{gm2}
Let $S$, $x(0)$, $\cU$, and $G=(V,E)$ be as in Game~\ref{gm1}. Let $m\in\NN$. At time $t=-1$, Player 2 chooses $U(d)\subseteq\cU$ for all $d\in [m]$ in such a way that $\{U(d)~|~d\in [m]\}$ is a partition of $\cU$. Then, each edge $(x,y)\in E$ is labeled as in \eqref{labe}. After this step, the game proceeds the same as Game \ref{gm1}.
\end{game}

Analogously to Proposition \ref{propo0}, it can be easily shown that the FPCP admits a solution if and only if Player 2 has a winning strategy in Game \ref{gm2}.

The problem of the existence of a winning strategy in Game \ref{gm2} can nominally be solved by reducing it to the problem of existence of a winning strategy in Game \ref{gm1}. Namely, every choice of a partition $\{U(d)~|~d\in [m]\}$ at time $t=-1$ generates a different instance of Game \ref{gm1}, so Player 2 has a winning strategy in Game \ref{gm2} if and only if there exists a partition $\{U(d)~|~d\in [m]\}$ for which Player 2 has a winning strategy in Game \ref{gm1}. However, an algorithm that determines a winning strategy for Game \ref{gm2} by considering all partitions $\{U(d)~|~d\in [m]\}$ is infeasible for large $\cU$, as the number of those partitions is not less than $m^{|\cU|-m}$ \cite{RenDob69}.

In the following section, we propose a graph-theoretical approach to the problem of determining the existence of winning strategies for Player 2 in the above games, resulting in easily computable conditions for the existence of a partition and a controller in the FPCP.

\section{Graph Labeling Problem}
\label{gra}

The previous section interprets system motion as a game on a labeled graph. By building upon this approach, we can convert the problem of finding a partition of the set of control inputs that admits a safe control policy --- the FPCP --- to an equivalent problem of labeling of graph edges.

\begin{theorem}
\label{propo1}
Let $S$, $x(0)$, $\cU$, and $G$ be as in Game~\ref{gm1}. The FPCP admits a solution if and only if there exist an induced subgraph $G_{\hat{S}}=(\hat{S},E_{\hat{S}})\subseteq G$ with $\hat{S}\subseteq S$ and a labeling $l~:~E_{\hat{S}}\to [m]$ such that the following properties hold:
\begin{enumerate}[leftmargin=23pt]
\item[(C1)] $x(0)\in \hat{S}$,
\item[(C2)] for all $x\in \hat{S}$, $$l\left(\{(x,x')\in E_{\hat{S}}~|~x'\in\hat{S}\}\right)=[m]\textrm{,}$$ and 
\item[(C3)] if $(x,y),(x',y')\in E_{\hat{S}}$ satisfy $y-x=y'-x'$, then $l(x,y)=l(x',y')$.
\end{enumerate}
\end{theorem}

\begin{proof}
As previously noted, the FPCP admits a solution if and only if there exists a winning strategy for Player 2 in  Game \ref{gm2}. Assume first that such a winning strategy exists, with the corresponding partition $\{U(d)~|~d\in [m]\}$ and a labeling $l:E\to [m]$ that satisfies \eqref{labe}. Let us now define $G_{\hat{S}}=(\hat{S},E_{\hat{S}})$ as the induced subgraph of $G$ with its vertex set $\hat{S}$ consisting of {\em all} the values that the system state $x(t)$ can possibly assume under the chosen winning strategy, for {\em all} potential input sequences $d:\NN_0\to [m]$. We claim that $G_{\hat{S}}$, with the labeling $l$ restricted to $E_{\hat{S}}$, satisfies (C1)--(C3).

First, since $\hat{S}$ is constructed from the winning strategy of Player 2, $x(0)\in\hat{S}\subseteq S$. Thus, (C1) holds. Property (C2) holds because, by definition of $\hat{S}$, for each $x\in \hat{S}$ there exists a $t\geq 0$ and a sequence $d(0),\ldots,d(t-1)$ such that $x(t)=x$, and for each $d'\in [m]$, setting $d(t)=d'$ requires that $l(x(t),x(t+1))=d'$. Property (C3) holds by \eqref{labe}.

In the other direction, assume that there exist an induced subgraph $G_{\hat{S}}$, $\hat{S}\subseteq S$, and a labeling function $l:E_{\hat{S}}\to[m]$ that satisfies (C1)--(C3). We will prove that the FPCP admits a solution.

Define 
\begin{equation}
\label{defur1}
\tilde{U}(d)=\left\{y-x~|~(x,y)\in E_{\hat{S}},\, l(x,y)=d\right\}
\end{equation}
for all $d\in \{1,\ldots,m\}$, and
\begin{equation}
\label{defur2}
\begin{split}
U(d)& =\tilde{U}(d) \textrm{ for all } d\leq m-1\textrm{,}\\
U(m)& =\tilde{U}(m)\bigcup\left(\cU\backslash\bigcup_{d=1}^{m-1}\hat{U}(d)\right)\textrm{.}
\end{split}
\end{equation}
Clearly, $\{U(1),\ldots,U(m)\}$ is a partition of $\cU$. We define $\tilde{l}:E\to [m]$ by \eqref{labe}, with $U(d)$ defined as in \eqref{defur1}--\eqref{defur2}. For any $(x,y)\in E_{\hat{S}}$, $\tilde{l}(x,y)=d$ if and only if $y-x\in U(d)$ by \eqref{labe}, which by \eqref{defur1}--\eqref{defur2} implies $l(x,y)=d$. Thus, $\tilde{l}$ and $l$ are the same on $E_{\hat{S}}$, so with a standard abuse of notation, we will refer to $\tilde{l}$ as $l$ in the remainder of the proof. 

Let us now define $\hat{u}:\hat{S}\times [m]\to\cU$ as any function with a following property: 
\begin{equation}
\label{defu}
\hat{u}(x,d)\in \{y-x~|~y\in \hat{S},\, (x,y)\in E_{\hat{S}},\, l(x,y)=d\}\textrm{.}
\end{equation}
We note that the existence of a function $\hat{u}$ that satisfies \eqref{defu} follows from (C2), although uniqueness is not guaranteed. 

We claim that any system run given by $x(t+1)=x(t)+\hat{u}(x(t),d(t))$ results in the system state remaining within $\hat{S}\subseteq S$, and that $\hat{u}(x(t),d(t))\in U(d(t))$ for all $t\in\NN_0$. For the claim that $x(t)\in \hat{S}$ for all $t$, we proceed by induction. By (C1), $x(0)\in \hat{S}$. Assume now that $x(t)\in \hat{S}$. Then, $x(t+1)=x(t)+\hat{u}(x(t),d(t))\in \hat{S}$ by \eqref{defu}. 

For the claim that $\hat{u}(x(t),d(t))\in U(d(t))$ for all $t$, we note that $l(x(t),x(t)+\hat{u}(x(t),d(t)))=d(t)$ by \eqref{defu}, so $\hat{u}(x(t),d(t))\in U(d(t))$ by \eqref{defur1}--\eqref{defur2}.

Thus, $\hat{u}$ is a solution to the RPCP for the partition $\{U(1),\ldots,U(m)\}$. Hence, the FPCP admits a solution.
\end{proof}
\begin{remark}
In the latter direction in the proof of Theorem \ref{propo1}, technically we constructed a memoryless policy $\hat{u}:\hat{S}\times [m]\to\cU$ instead of a memory-conscious policy $\hat{u}:\cup_{i=1}^\infty [m]^i\to\cU$ as required in the RPCP. Thus, Theorem \ref{propo1} also shows that Game \ref{gm1} and Game \ref{gm2} admit a winning strategy for Player 2 if and only if they admit a {\em memoryless} winning strategy, which was also discussed in \cite{DoyRas11}.
\end{remark}

With Theorem \ref{propo1} in mind, the FPCP can be transformed into the following problem.

{\em Invariant subgraph labeling problem (ISLP):} 
Let $S$, $x(0)$, $\cU$, $m$, and $G$ be as in Game \ref{gm1}. Let $m\in\NN$. Determine whether there exist an induced subgraph $G_{\hat{S}}=(\hat{S},E_{\hat{S}})\subseteq G$ with $\hat{S}\subseteq S$ and a labeling $l:E_{\hat{S}}\to [m]$ which satisfy (C1)--(C3).

Let us briefly note that if one was to omit requiring (C3) from the ISLP, such a problem reduces to finding an induced subgraph $G_{\hat{S}}\subseteq G$ with $x(0)\in\hat{S}\subseteq S$ and $\mindeg(G_{\hat{S}})\geq m$. This problem is a variant of the {\em minimum subgraph of minimum degree} problem; see, e.g., \cite{Amietal12} and the references therein. 
We now proceed to determine sufficient and necessary conditions for the ISLP to admit a solution.

\section{Conditions for a Good Labeling}
\label{condit}

As discussed above, property (C2) in Theorem \ref{propo1} trivially imposes a simple necessary condition for the ISLP to admit a solution.

\begin{proposition}
\label{lem1}
If there exist an induced subgraph $G_{\hat{S}}$ and a labeling $l$ satisfying the conditions of ISLP, then $$\mindeg(G_{\hat{S}})\geq m\textrm{.}$$
\end{proposition}

The condition given in Proposition \ref{lem1} is not sufficient for existence of a labeling satisfying the conditions of the ISLP. The following example gives an induced subgraph $G_{\hat{S}}\subseteq G$ with $\mindeg(G_{\hat{S}})\geq m$ such that no labeling $l:E_{\hat{S}}\to [m]$ satisfies (C2)--(C3).

\begin{example}
\label{exn}
Consider $n=2$, $m=3$, $x(0)=0$, $S=\{x\in\ZZ^2~|~\|x\|_1\leq 1\}$, and $U=\{u\in\ZZ^2~|~\|u\|_\infty=1\}$. Let $\hat{S}=S$. Clearly, $x(0)\in\hat{S}$, and, as illustrated in Fig.~\ref{fig2}, $\mindeg(G_{\hat{S}})\geq m$. Nonetheless, $S$ does not admit a labeling satisfying both (C2) and (C3). Assume otherwise. Let $l:~E_{\hat{S}}\to [m]$ be such a labeling. By (C3), $l$ is translation-invariant. Thus, we denote by $\hat{l}(1)$ the label of all edges that point north (i.e., $(x,y)\in E_{\hat{S}}$ such that $y-x=(0,1)$), $\hat{l}(2)$ the label of NE edges ($(x,y)\in E_{\hat{S}}$ such that $y-x=(1,1)$), $\hat{l}(3)$ for E edges, etc. By applying (C2) to 
\begin{enumerate}
\item[(i)] vertices $(0,-1)$, $(-1,0)$, $(0,1)$, and $(1,0)$, respectively, we can conclude that, for each $k\in\{0,1,2,3\}$, $\hat{l}(2k)$, $\hat{l}(2k+1)$, and $\hat{l}(2k+2)$ need to be all different (for ease of notation, we identify $\hat{l}(0)$ with $\hat{l}(8)$),
\item[(ii)] vertex $(0,0)$, we note that $\hat{l}(1)$, $\hat{l}(3)$, $\hat{l}(5)$, and $\hat{l}(7)$ need to have three different values.
\end{enumerate}

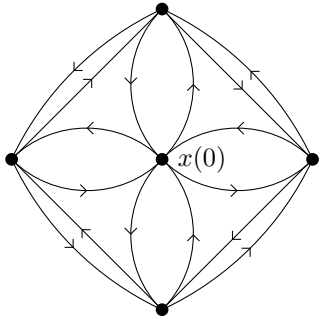
\begin{figure}[ht]
\centering
\begin{tikzpicture}[scale=2]
\tikzstyle{point}=[thick,fill=black, draw=black,circle,inner sep=0pt,minimum width=4pt,minimum height=4pt]
\tikzset{->-/.style={decoration={
  markings,
  mark=at position .5 with {\arrow{Straight Barb[]}}},postaction={decorate}}}
\tikzset{-<-/.style={decoration={
  markings,
  mark=at position .5 with {\arrow{Straight Barb[reversed]}}},postaction={decorate}}}
\node[point, label=0:{$x(0)$}] at (0,0) {};
\node at (0,1.15) {};
\node[point] at (0,1) {};
\node[point] at (-1,0) {};
\node[point] at (1,0) {};
\node[point] at (0,-1) {};
\draw[->-] (0,0) arc (45:135:0.71);
\draw[->-] (-1,0) arc (225:315:0.71);
\draw[->-] (1,0) arc (45:135:0.71);
\draw[->-] (0,0) arc (225:315:0.71);
\draw[->-] (0,1) arc (135:225:0.71);
\draw[->-] (0,0) arc (-45:45:0.71);
\draw[->-] (0,0) arc (135:225:0.71);
\draw[->-] (0,-1) arc (-45:45:0.71);
\draw[-<-] (0,1) -- (-1,0);
\draw[->-] (0,1) arc (45+70:225-70:2.07);
\draw[-<-] (-1,0) -- (0,-1);
\draw[->-] (-1,0) arc (135+70:315-70:2.07);
\draw[-<-] (0,-1) -- (1,0);
\draw[->-] (0,-1) arc (-135+70:45-70:2.07);
\draw[-<-] (1,0) -- (0,1);
\draw[->-] (1,0) arc (-45+70:135-70:2.07);

\end{tikzpicture}
\caption{An illustration of Example \ref{exn}. The vertices of $\hat{S}=S$ and corresponding directed edges of $E_{\hat{S}}$ are denoted in black.}
\label{fig2}
\end{figure}

Now, from (ii), assume without loss of generality that $\hat{l}(1)=1$, $\hat{l}(3)=2$, and $\hat{l}(5)=3$. Then, by (i) for $k=0$ and $k=1$, $\{\hat{l}(8),\hat{l}(2)\}=\{2,3\}$ and $\{\hat{l}(2),\hat{l}(4)\}=\{1,3\}$. Hence, $\hat{l}(2)=3$, $\hat{l}(8)=2$, and $\hat{l}(4)=1$. Since $\{\hat{l}(4),\hat{l}(6)\}=\{1,2\}$ by (i) for $k=2$, we have $\hat{l}(6)=2$. Thus, $\hat{l}(8)=\hat{l}(6)$, which is in contradiction with (i) for $k=3$.
\hfill $\square$
\end{example}

Even though Proposition \ref{lem1} only gives a necessary condition for the ISLP to admit a solution, there does exist a related sufficient condition. Namely, if there exists an induced subgraph $G_{\hat{S}}$ with large enough $\mindeg(G_{\hat{S}})$, then there exists a labeling of $E_{\hat{S}}$ which solves the FPCP. We prove such a result using the probabilistic method (see, e.g., \cite{Erd59, Erd61, AloSpe08} for more details).

\begin{theorem}
\label{thmbi}
Let $S$, $x(0)$, $\cU$, and $G=(V,E)$ be as in Game~\ref{gm1}. If there exists a finite induced subgraph $G_{\hat{S}}=(\hat{S},E_{\hat{S}})\subseteq G$ with $x(0)\in\hat{S}\subseteq S$ and \begin{equation}
\label{bou}
\mindeg(G_{\hat{S}})\geq m\ln\left(m|\hat{S}|\right)\textrm{,}
\end{equation}
then there exists a labeling $l:E_{\hat{S}}\to [m]$ such that $G_{\hat{S}}$ and $l$ satisfy properties (C1)--(C3).
\end{theorem}
\begin{proof}
Let us label each element $u\in\cU$ by  $\hat{l}(u)\in [m]$, where each label is chosen independently and uniformly. We define $l:E_{\hat{S}}\to [m]$ by $l(x,y)=\hat{l}(y-x)$.
By definition of $l$, (C3) is satisfied. Property (C1) is also satisfied by the theorem assumptions. 

Let $B$ be the event that the label $l$ does not satisfy (C2), i.e., that there exists a vertex $x\in\hat{S}$ such that 
\begin{equation}
\label{aux1}
l\left(\{(x,x')\in E_{\hat{S}}~|~x'\in\hat{S}\}\right)\neq [m]\textrm{.}
\end{equation} Define $B_x$ as the event that $l$ satisfies \eqref{aux1} for the particular $x\in\hat{S}$. In particular, define $B_x^i$ as the event that $i\notin l(\{(x,x')\in E_{\hat{S}}~|~x'\in\hat{S}\})$. 

If we can show that $\Pr(B)<1$, this will mean that there exists at least one labeling $l$ such that $B$ does {\em not} occur, i.e., that (C1)--(C3) are all satisfied.

By the definitions of $B_x$ and $B_x^i$ and the union bound \cite{Ven12}, we obtain $$\Pr(B)\geq\sum_{x\in\hat{S}}\Pr(B_x)\geq \sum_{\substack{x\in\hat{S} \\
i \in [m]}}\Pr(B_x^i)\textrm{.}$$ Hence, if we show that 
\begin{equation}
\label{aux2}
\Pr(B_x^i)<1/(m|\hat{S}|)
\end{equation}
holds for all $x\in\hat{S}$ and $i\in [m]$, we are done.

Consider the event $B_x^i$ for fixed $x\in\hat{S}$ and $i\in [m]$. For each edge $(x,x')\in E_{\hat{S}}$, $x'-x\in\cU$ is different. Thus, labels $l(x,x')$ have been chosen uniformly and independently. Hence, \begin{equation*}
\begin{split}\Pr(B_x^i)& =\Pr\left(l(x,x')\neq i \textrm{ for all } (x,x')\in E_{\hat{S}}\right) \\ & =\prod_{(x,x')\in E_{\hat{S}}}\Pr\left(l(x,x')\neq i\right)=\prod_{(x,x')\in E_{\hat{S}}} (1-1/m)\textrm{.}
\end{split}
\end{equation*}
Thus, $\Pr(B_x^i)=(1-1/m)^{\deg_{G_{\hat{S}}}(x)}\leq (1-1/m)^{\mindeg(G_{\hat{S}})}$. By simply noting that $(1-1/m)^m<e^{-1}$ (see, e.g., \cite{Mol12}), we obtain $\Pr(B_x^i)\leq (1-1/m)^{\mindeg(G_{\hat{S}})}<e^{-\mindeg(G_{\hat{S}})/m}$. We now obtain \eqref{aux2} from \eqref{bou}.
\end{proof}

Theorem \ref{thmbi} gives a condition for solving the ISLP, i.e., the FPCP, based on finding a suitable subset $\hat{S}$ of the safe set. One way of producing such a subset is by finding a sufficiently dense subgraph of $S$, with a suitable definition of density. In the interest of brevity, we omit further details. We provide two illustrative examples of determining $\hat{S}$ in Section \ref{exams}.

Returning to the running example, construction on the right side of Fig.~\ref{runn}, where $\mindeg(G_{\hat{S}})=m<m\ln (m|\hat{S}|)$, shows that the condition expressed in Theorem \ref{thmbi} is not necessary for the solvability of the FPCP. We will return to this example in Section \ref{exams}, where we provide some intuition for the ``reason'' that it yields a solution to the FPCP, even though it does not satisfy the sufficient condition expressed in Theorem~\ref{thmbi}.

\section{Efficient Labeling Algorithm}
\label{ela}

The proof of Theorem \ref{thmbi} does not provide a mechanism for constructing a good labeling. Instead, it merely states that a uniformly chosen labeling will solve the FPCP with probability $1-\Pr(B)\geq 1-m|\hat{S}|(1-1/m)^{\mindeg(G_{\hat{S}})}$. Thus, an algorithm that randomly chooses labelings until it reaches one that solves the FPCP is going to have expected computational complexity no greater than $$O\left(\frac{|E_{\hat{S}}|}{1-m|\hat{S}|(1-1/m)^{\mindeg(G_{\hat{S}})}}\right)\textrm{,}$$ assuming that a random draw is performed in $O(1)$ time, and including the time to verify whether a labeling satisfies (C2). Thus, if $m|\hat{S}|(1-1/m)^{\mindeg(G_{\hat{S}})}\approx 1$, a randomized algorithm might take a substantial amount of time to finish.

We now present an alternative deterministic algorithm that produces a correct labeling in $O(|E_{\hat{S}}|+|\cU|m|\hat{S}|)$ operations.

\begin{algorithm}
\label{alg}
Let $\cU=\{u_1,\ldots,u_{|\cU|}\}$. Define a labeling $\hat{l}$ on $\cU$ inductively as follows. Let \begin{equation}
\begin{split}
\label{defli}
l_i\in\argmin_{l'\in [m]}\sum_{\substack{x\in\hat{S} \\
j \in [m]}}& \Pr(B_x^j~|~\hat{l}(u_1)=l_1,\ldots, \\
&\ldots,\hat{l}(u_{i-1})=l_{i-1},\hat{l}(u_i)=l')
\end{split}
\end{equation}
and define $\hat{l}(u_i)=l_i$ for $i=1,2,\ldots,|\cU|$, where labeling $l~:~E_{\hat{S}}~\to~[m]$ is given by $l(x,y)=\hat{l}(y-x)$ for all $(x,y)\in E_{\hat{S}}$.
\end{algorithm}

\begin{theorem}
Assume that \eqref{bou} holds for a finite induced subgraph $G_{\hat{S}}$ with $x(0)\in\hat{S}\subseteq S$. Let $\hat{l}$ and $l$ be defined as in Algorithm \ref{alg}. Then, $G_{\hat{S}}$ and $l$ satisfy properties (C1)--(C3).
\end{theorem}
\begin{proof}
By \eqref{defli}, for each $i\in[|\cU|]$, \begin{equation*}
\begin{split}& \sum_{\substack{x\in\hat{S} \\
j \in [m]}}\Pr(B_x^j~|~\hat{l}(u_1)=l_1, \ldots, \hat{l}(u_i)=l_i)\leq \\ & \sum_{k\in[m]}\frac{1}{m}\sum_{\substack{x\in\hat{S} \\
j \in [m]}}\Pr(B_x^j~|~\hat{l}(u_1)=l_1, \ldots, \hat{l}(u_i)=k)= \\
& \sum_{\substack{x\in\hat{S} \\
j \in [m]}}\sum_{k\in[m]}\frac{1}{m}\Pr(B_x^j~|~\hat{l}(u_1)=l_1, \ldots, \hat{l}(u_i)=k)\leq 
\\
& \sum_{\substack{x\in\hat{S} \\
j \in [m]}}\Pr(B_x^j~|~\hat{l}(u_1)=l_1, \ldots, \hat{l}(u_{i-1})=l_{i-1})\textrm{.} 
\end{split}
\end{equation*}
Hence, inductively,
\begin{equation}
\label{bom}
\begin{split}
\sum_{\substack{x\in\hat{S} \\
j \in [m]}}\Pr(B_x^j~|~\hat{l}(u_1)=l_1,& \ldots, \hat{l}(u_{|\cU|})=l_{|\cU|}) \\ \leq\sum_{\substack{x\in\hat{S} \\
j \in [m]}}\Pr(B_x^j)<1\textrm{,}
\end{split}
\end{equation} where the last inequality holds by the proof of Theorem \ref{thmbi}. On the other hand, $l$ is entirely defined by $\hat{l}(u_1), \ldots, \hat{l}(u_{|\cU|})$. Hence, $\Pr(B_x^j~|~\hat{l}(u_1)=l_1,  \ldots, \hat{l}(u_{|\cU|})=l_{|\cU|})$ equals either $0$ or $1$ for each $x\in\hat{S}$, $j\in[m]$. By \eqref{bom}, we thus have $\Pr(B_x^j~|~\hat{l}(u_1)=l_1,  \ldots, \hat{l}(u_{|\cU|})=l_{|\cU|})=0$ for all $x\in\hat{S}$, $j\in[m]$, i.e., $G_{\hat{S}}$ and $l$ satisfy the conditions of the ISLP.
\end{proof}
\begin{proposition}
Algorithm \ref{alg} can be performed in $O(|E_{\hat{S}}|+|\cU|m|\hat{S}|)$ operations.
\end{proposition}
\begin{proof}
Clearly, the computational complexity of Algorithm \ref{alg} depends on the complexity of solving the optimization problem in \eqref{defli} for each $i\in[|\cU|]$. For each $i\in[|\cU|]$, $x\in\hat{S}$, and $j\in[m]$, if $j=l_k$ for some $k\in\{1,\ldots,i\}$ and $x+u_k\in\hat{S}$, then $\Pr(B_x^j~|~\hat{l}(u_1)=l_1,  \ldots, \hat{l}(u_i)=l_i)=0$. Otherwise, \begin{equation*}
\begin{split}\Pr(B_x^j~|~\hat{l}(u_1)=l_1, & \ldots, \hat{l}(u_i)=l_i)= \\
& (1-1/m)^{|\{i<k\leq |\cU|~|~x+u_k\in\hat{S}\}|}\textrm{.}
\end{split}
\end{equation*}
Thus, if we precompute whether $x+u_k\in\hat{S}$ for each $x\in\hat{S}$ and $k\in[|\cU|]$, and all values $|\{i<k\leq |\cU|~|~x+u_k\in\hat{S}\}|$, which can be performed in $O(E_{\hat{S}})$ operations, computing \eqref{defli} can be performed in $m|\hat{S}|$ time for each $i\in\{1,\ldots,|\cU|\}$, by merely updating all $\Pr(B_x^j~|~\hat{l}(u_1)=l_1,  \ldots, \hat{l}(u_i)=l_i)$ at the end of step $i$. Hence, Algorithm \ref{alg} indeed operates in $O(E_{\hat{S}}+|\cU|m|\hat{S}|)$ time.
\end{proof}

Provided with a labeling $l:E_{\tilde{S}}\to[m]$, given in Algorithm \ref{alg}, the system design and control policy which solve the FPCP are given by \eqref{defur1}--\eqref{defur2} and \eqref{defu}. We now proceed to illustrate the obtained results on two practical scenarios.

\section{Examples}
\label{exams}

\subsection{Damaged Vehicle}
Having given conditions for solvability of the RPCP and the FPCP, we return to our running example. Let us consider a vehicle operating on $V=\ZZ^n$, with the ability to either move along the coordinate axes or stay in place, i.e., $\cU=\{0,\pm e_1,\ldots,\pm e_n\}$. Naturally, only $n\leq 3$ makes direct physical sense. A similar example has been considered in the context of safety games in \cite{NelTop16}. However, in that paper the agent and the adversary alternate in taking control of the vehicle, and the focus of the paper was on efficient computation of safe control policies for a given system design, and not on determining a good system design.

The safety objective that we consider is that the vehicle remains close to its initial position $x(0)=0$, i.e., $S=\{x~|~\|x\|_\infty\leq k\}$ for some $k\in\NN_0$. As we showed in Example \ref{exrun}, there exists a safe system design for $n=2$, $m=2$, and $k=1$. In this section, we are interested in discussing the maximal loss of control that still enables a safe system design, i.e., for a given $n$ and $k$, the maximal $m$ such that the FPCP admits a solution.

It is clear that if $k=0$, the agent cannot afford any loss of authority, i.e., the only acceptable $m$ equals $1$. If $k\geq 1$, we claim that the maximal $m$ equals $n+1$. 

Let us first show that the FPCP has a solution for $m=n+1$. A partition $\{U_1,\ldots,U_{n+1}\}$ that admits a solution to the RPCP is given by $U_i=\{e_i,-e_i\}$ for $i\leq n$, and $U_{n+1}=\{0\}$. Indeed, analogously to the construction on the right side of Fig.~\ref{runn}, a control policy which alternately chooses $e_d$ and $-e_d$ every time the adversary chooses input $d\in [n]$, and $0$ if the adversary chooses $d=n+1$, results in the agent's state always remaining in $\hat{S}=\{x~|~\|x\|_\infty\leq 1\}$.

On the other hand, if $m\geq n+2$, since there is a total of $2n$ non-zero elements in $\cU$, for any partition $\{U_1,\ldots,U_m\}$, some partition element $U_j$ will equal $\{e_i\}$ or $\{-e_i\}$ for some $i$. However, by then repeatedly choosing $d(t)=j$, the adversary can be assured that $\|x(t)\|_\infty=t$, i.e., $\|x\|_\infty>k$ after finitely many steps. Thus, the maximal value of $m$ for which the FPCP admits a solution is indeed $n+1$.

If $m=n+1$ and $\hat{S}=S=\{x~|~\|x\|_\infty\leq 1\}$, sufficient condition \eqref{bou} from Theorem \ref{thmbi} does not hold, as $\mindeg(G_{\hat{S}})=m<m\ln(m|\hat{S}|)$. Nonetheless, the solution to the FPCP exists. Let us briefly discuss this gap between sufficiency and necessity of condition \eqref{bou}. The proof of Theorem \ref{thmbi} relies on some degree of genericity of a correct labeling, i.e., a positive probability that a randomly chosen labeling will be correct. On the other hand, the solution to the FPCP when $m=n+1$ is highly structured. Namely, each element of $\{U_1,\ldots,U_m\}$ needs to equal $\{0\}$ or $\{e_i,-e_i\}$ for some $i$. Otherwise, there will exist $U_j$ that equals $\{e_i\}$ or $\{-e_i\}$ for some $i$, and by repeating $d(t)=j$, the adversary will be able to force the system state to move arbitrarily far away from $x(0)$. Hence, the partition that yields a solution to the RPCP is in fact unique up to a permutation: $U(i)=\{e_{i},-e_{i}\}$ for all $i\leq n$, and $U(n+1)=\{0\}$. Thus, as $n$ increases, the probability of a uniformly chosen partition yielding a solution to the RPCP tends to $0$.

\subsection{Communication over a Channel}
We now move from the setting of damaged autonomous systems to that of user-responsive systems. Consider the framework --- originally introduced in \cite{Sha48} --- where, at every time $t$, a message chosen from some finite message set $\cM$, $|\cM|=m$, is sent over a communication channel. Each message is encoded as a bit-string (i.e., {\em codeword}) of some fixed length $n$. This codeword does not need to be the same every time the same message is sent; there could be multiple ways to communicate the same message. However, two different messages cannot be encoded in the same way.

The {\em running digital sum} (RDS) $x(t)$ is defined as the vector consisting of differences in the number of $1$'s and $0$'s that were sent in each coordinate of the bit-string until time $t$. Thus, $x(t)$ satisfies \eqref{dtd} for $x(0)=0$, where $u(t)$ is an encoding of the message passed at time $t$, with zeros in the bit-string replaced by $-1$'s, and $\cU=\{-1,1\}^n$ \cite{CohLit91}. An illustration of such a system for $n=2$ is given in  Fig.~\ref{fig}.

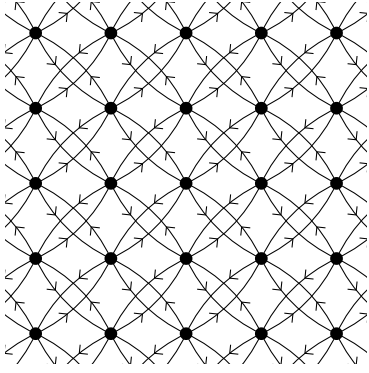
\begin{figure}[ht]
\centering
\begin{tikzpicture}[scale=1]
\clip (-2.4,-2.4) rectangle (2.4,2.4);
\tikzstyle{point}=[thick,fill=black, draw=black,circle,inner sep=0pt,minimum width=4pt,minimum height=4pt]
\tikzset{->-/.style={decoration={
  markings,
  mark=at position .35 with {\arrow{Straight Barb[]}}},postaction={decorate}}}
\tikzset{-<-/.style={decoration={
  markings,
  mark=at position .35 with {\arrow{Straight Barb[reversed]}}},postaction={decorate}}}

    \foreach \x in {-3,...,3}
    \foreach \y in {-3,...,3}
    {
    \node[point] at (\x,\y) {};
    \draw[->-] (\x,\y) arc(-45-20:-45+20:2.07);
    \draw[->-] (\x,\y) arc(45-20:45+20:2.07);
        \draw[->-] (\x,\y) arc(135-20:135+20:2.07);
    \draw[->-] (\x,\y) arc(225-20:225+20:2.07);  
    }
\end{tikzpicture}
\caption{An illustration of the dynamical system that describes the RDS. The vertices of $G$ and the corresponding directed edges of $E$ are denoted in black.}
\label{fig}
\end{figure}

Encoding policies for which the RDS in a channel remains small regardless of the passed messages naturally reduce the effects of various categories of noise \cite{CohLit91}, \cite{Sch04}. Since encodings of different messages are pairwise disjoint, the problem of constructing encoding policies with bounded RDS can be naturally interpreted as the FPCP, with the safe set $S=\{x~|~\|x\|_\infty\leq k\}$. In this section, we are primarily interested in finding the smallest codeword length $n$ such that there exists an encoding policy for which the RDS remains within $S$.

For $k=0$, there clearly does not exist $n$ which yields a solution for the RPCP. For $k=1$, the only $m$ for which there exists an $n$ which yields a solution for the RPCP is $m=1$, and in that case $n=1$ suffices. For $k\geq 2$, a bound on $n$ can be obtained from Theorem \ref{thmbi} as follows. 

\begin{proposition}
\label{line}
Let $m,n\in\NN$, $\cU=\{-1,1\}^n$, and $x(0)=0$. Then, if $n\geq 3\max(\log_2 m,11)$, the FPCP admits a solution for $S=\{x\in\ZZ^n~|~\|x\|_\infty\leq 2\}$.
\end{proposition}
\begin{proof}
Let us define $\hat{S}=V_1\cup V_2$, where $V_1=\{-1,1\}^n$, and $V_2=\{(x_1,\ldots,x_n)\in\{-2,0,2\}^n~|~x_i=0 \textrm{ for at least } n/2\textrm{ } i\textrm{'s}\}$. We note that $x(0)\in\hat{S}\subseteq S$. 

Let us examine the outgoing degree $\deg_{G_{\hat{S}}}(v)$ of every vertex $v\in\hat{S}$ in the induced subgraph $G_{\hat{S}}\subseteq G$. If $v=(v_1,v_2,\ldots,v_n)\in V_1$, then 
\begin{equation}
\label{deg1}
\deg_{G_{\hat{S}}}(v)=\sum_{i\geq n/2}\binom{n}{i}\geq 2^{n-1}\textrm{,}
\end{equation}
as the set of neighbors of $v$ is given by all vertices $\overline{v}=(\overline{v}_1,\ldots,\overline{v}_n)\in\ZZ^n$ that satisfy (i) $\overline{v}_i\in\{0,2v_i\}$ for all $i\in[n]$, and (ii) $\overline{v}_i=0$ for at least $n/2$ $i$'s. If $v\in V_2$, then
\begin{equation}
\label{deg2}\deg_{G_{\hat{S}}}(v)\geq 2^{n/2}\textrm{,}
\end{equation}
as the set of neighbors of $v$ is given by all $\overline{v}\in\ZZ^n$ that satisfy $\overline{v}_i\in\{-1,1\}$ if $v_i=0$, and $\overline{v}_i=v_i/2$ otherwise. Thus, from \eqref{deg1} and \eqref{deg2}, we obtain $\mindeg(G_{\hat{S}})\geq 2^{n/2}$.

We note that $|\hat{S}|=|V_1|+|V_2|\leq 2^n+3^n\leq 3^{n+1}$. Thus, $m\ln(m|\hat{S}|)\leq m\ln m+m(n+1)\ln 3\leq n2^{n/3}\ln(2)/3+(n+1)2^{n/3}\ln (3)$. It can be shown that $n2^{n/3}\ln(2)/3+(n+1)2^{n/3}\ln (3)\leq 2^{n/2}$ for all $n\geq 33$. Thus, the conditions of Theorem \ref{thmbi} are satisfied.
\end{proof}
An illustration of the construction of $\tilde{S}$ used in the proof of Proposition \ref{line} is given in Fig.~\ref{figlin}, for $m=n=2$. We note that Fig.~\ref{figlin} shows that it is possible to construct a labeling (i.e., partition $\{U_1,U_2\}$) even for $n=2$, indicating that the bound in Proposition \ref{line} is very liberal.

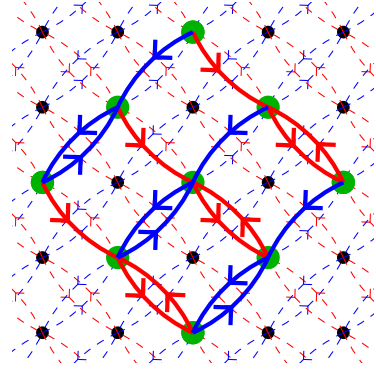
\begin{figure}[ht]
\centering
\begin{tikzpicture}[scale=1]
\clip (-2.4,-2.4) rectangle (2.4,2.4);
\tikzstyle{point}=[thick,fill=black, draw=black,circle,inner sep=0pt,minimum width=8pt,minimum height=8pt]
\tikzset{->-/.style={decoration={
  markings,
  mark=at position .45 with {\arrow{Straight Barb[]}}},postaction={decorate}}}
\tikzset{-<-/.style={decoration={
  markings,
  mark=at position .45 with {\arrow{Straight Barb[reversed]}}},postaction={decorate}}}

    \foreach \x in {-3,...,3}
    \foreach \y in {-3,...,3}
    {
    \node[point,minimum width=4pt,minimum height=4pt] at (\x,\y) {};
    \draw[->-,blue,dashed] (\x,\y) arc(-45-20:-45+20:2.07);
    \draw[->-,red,dashed] (\x,\y) arc(45-20:45+20:2.07);
        \draw[->-,blue,dashed] (\x,\y) arc(135-20:135+20:2.07);
    \draw[->-,red,dashed] (\x,\y) arc(225-20:225+20:2.07);  
    }
    \foreach \x in {0,...,2}
    {
    \node[point,green!70!black] at (\x,2-\x) {};
    \node[point,green!70!black] at (-\x,2-\x) {};
    \node[point,green!70!black] at (\x,-2+\x) {};  
    \node[point,green!70!black] at (-\x,-2+\x) {};
    }
    \node[point,green!70!black] at (0,0) {};
    
     \draw[->-,red, ultra thick] (0,2) arc(225-20:225+20:2.07);
     \draw[->-,red, ultra thick] (1,1) arc(225-20:225+20:2.07);
     \draw[->-,red, ultra thick] (-1,1) arc(225-20:225+20:2.07);
     \draw[->-,red, ultra thick] (0,0) arc(225-20:225+20:2.07);
     \draw[->-,red, ultra thick] (-2,0) arc(225-20:225+20:2.07);
     \draw[->-,red, ultra thick] (-1,-1) arc(225-20:225+20:2.07);
     \draw[->-,red, ultra thick] (0,-2) arc(45-20:45+20:2.07);
  	\draw[->-,red, ultra thick] (1,-1) arc(45-20:45+20:2.07);
	\draw[->-,red, ultra thick] (2,0) arc(45-20:45+20:2.07);
\draw[->-,blue, ultra thick] (0,2) arc(135-20:135+20:2.07);    
\draw[->-,blue, ultra thick] (-1,1) arc(135-20:135+20:2.07);    
\draw[->-,blue, ultra thick] (1,1) arc(135-20:135+20:2.07);    
\draw[->-,blue, ultra thick] (0,0) arc(135-20:135+20:2.07);     
\draw[->-,blue, ultra thick] (2,0) arc(135-20:135+20:2.07);    
\draw[->-,blue, ultra thick] (1,-1) arc(135-20:135+20:2.07);    
\draw[->-,blue, ultra thick] (-2,0) arc(315-20:315+20:2.07);      
\draw[->-,blue, ultra thick] (-1,-1) arc(315-20:315+20:2.07);     
\draw[->-,blue, ultra thick] (0,-2) arc(315-20:315+20:2.07);    

\end{tikzpicture}
\caption{An illustration of a safe labeling for the RDS with $m=2$, $n=2$. Set $\hat{S}$ is denoted in green. A control policy on $\hat{S}$ that ensures safety is described by thicker arrows. We note that each element in $\hat{S}$ has an outgoing thick arrow in each color pointing into $\hat{S}$. Hence, the system state controlled by such a law will always remain within $\hat{S}$, for any $x(0)\in\hat{S}$.}
\label{figlin}
\end{figure}

As it is necessary to use codewords (i.e., bit-strings) of length at least $\lceil\log_2 m\rceil$ to distinguish between $m$ different messages, Proposition \ref{line} states that, if we use three times as many bits as necessary, we can ensure that the RDS stays within the smallest possible bounds. We remark that from the proof of Proposition \ref{line} it is clear that $n\geq 3\max(\log_2 m, 11)$ can be replaced by $n\geq (2+\varepsilon)\max(\log_2 m, n_\varepsilon)$ for any $\varepsilon\geq 0$, where $n_\varepsilon\to\infty$ as $\varepsilon\to 0$.

\section{Conclusion and Future Work}
\label{conc}

This paper presents a preliminary discussion on control, design, and motion planning abilities of an autonomous system where the controller experienced a partial loss of control authority. The paper is primarily interested in developing sufficient and necessary conditions for existence of a safe control policy in such a partly controlled system. In order to obtain these conditions, we interpreted the system motion as a variant of an adversarial safety game on a graph, where one of the player's moves is to label the edges of the game graph. We showed that the safety objective in the original control system is attainable if and only if such a game has a winning strategy, and showed that the game has a winning strategy if and only if there exists a labeling of the game graph that satisfies particular properties. We found a sufficient condition and a necessary condition for the existence of such a labeling in terms of minimal degrees of a subgraph of the original graph, and discussed how those conditions apply to the motion of an autonomous vehicle operating on an $n$-dimensional surface and to communication using a set of codewords of length $n$ with a bounded running digital sum.

The primary avenue of future work is in broadening the scope of the considered framework. In addition to discussing system dynamics more general than \eqref{dtd} --- which may be achieved by considering two-stage motions on a graph, one stage being involuntary ("drift"), and the other resulting from the performed actions --- it is meaningful to consider a broader class of control specifications, rather than solely safety. In general, tasks for autonomous systems are often expressed by a temporal logic specification (e.g., ``visit area $A$ infinitely many times, never go into area $B$, and eventually reach area $C$''). Previous work on designing provably correct control policies --- i.e., policies that are guaranteed to result in the system behavior satisfying a temporal logic specification --- primarily deals with systems whose control abilities are not compromised; see \cite{BaiKat08} for a thorough study. While there is a substantial body of work (see, e.g., \cite{Kupetal01} and the references therein) on systems whose control originally introduced in \cite{Sha48}, may depend on the environment, procedures for determining provably correct control policies for such systems are computationally complex. Providing simple graph-based criteria for existence of a system design that admits a correct control policy would present a significant next step towards ensuring system resilience under partial loss of control authority.

\bibliographystyle{IEEEtran}
\bibliography{refs}

\end{document}